\documentclass[12pt,a4paper]{article}
\usepackage{amsmath}
\usepackage{amsfonts}
\usepackage{amssymb}
\usepackage{amsthm}
\usepackage{dsfont}
\usepackage{paralist}
\usepackage{natbib}
\usepackage{bm}

\newtheoremstyle{myprop}
  {3pt}
  {3pt}
  {\itshape}
  {}
  {\scshape}
  {.}
  {.5em}
  {}
\theoremstyle{myprop}
\newtheorem{proposition}{Proposition}
\newtheorem{lemma}[proposition]{Lemma}
\newtheorem{theorem}[proposition]{Theorem}

\theoremstyle{remark}

\newcommand\N{\mathbb{N}}

\newcommand\R{\mathbb{R}}
\newcommand\Z{\mathbb{Z}}
\newcommand\eps{\varepsilon}
\newcommand{\vect}{\bm }  
\newcommand\Prob{\mathbb{P}}    
\newcommand\ind{\mathbb{I}}     

\newcommand\Fc{\mathcal{F}}

\newcommand\Bb{\mathbb{B}}

\newcommand\Db{\mathbb{D}}
\newcommand\Eb{\mathbb{E}}

\DeclareMathOperator{\Cov}{Cov}

\newcommand\weak{\rightsquigarrow}

\newcommand{\ip}[1]{\lfloor #1 \rfloor}

\allowdisplaybreaks[1]

\begin{document}

\title{A note on weak convergence of the sequential multivariate  empirical process under strong mixing}

\author{Axel B\"ucher\footnote{Universit\'{e} catholique de Louvain,
Institut de statistique, Voie du Roman Pays 20, 1348 Louvain-la-Neuve, Belgium. E-mail: axel.buecher@rub.de.} \smallskip\\ 
\textit{Universit\'{e} catholique de Louvain \&  Ruhr-Universit\"at Bochum} \\
}

\maketitle

\begin{abstract}
This article investigates weak convergence of the sequential $d$-dimensional empirical process under strong mixing. Weak convergence is established for mixing rates $\alpha_n = O(n^{-a})$, where $a>1$, which slightly improves upon existing results in the literature that are based on mixing rates depending on the dimension $d$.
\end{abstract}



\smallskip


\section{Introduction}
\def\theequation{1.\arabic{equation}}
\setcounter{equation}{0}

Let $(\vect U_i)_{i\in\Z}$, $\vect U_i = (U_{i1}, \dots, U_{id})$, be a strictly stationary sequence of $d$-dimensional random vectors whose marginals are standard uniform. Denote the joint cumulative distribution function of $\vect U_i$ by $C$. The corresponding empirical process is defined, for any $\vect u = (u_1, \dots, u_d) \in [0,1]^d$, by
\[
	\Db_n(\vect u) = \frac{1}{\sqrt n} \sum_{i=1}^n \{ \ind( \vect U_i \le \vect u) - C(\vect u) \}.
\]
Under various types of weak dependence conditions, the process $\Db_n$ is known to converge weakly in the space of bounded functions on $[0,1]^d$ equipped with the supremum-norm, denoted by $(\ell^\infty([0,1]^d), \| \cdot \|_\infty)$, to a tight, centered Gaussian process~$\Db_C$ with covariance
\[
	 \Cov\{  \Db_C(\vect u) , \Db_C(\vect v) \} = \sum_{i\in\Z} \Cov\{ \ind(\vect U_0\le \vect u) , \ind(\vect U_i\le \vect v) \},
\]
see for instance \cite{ArcYu94} and \cite{DouMasRio95} for $\beta$-mixing, \cite{Rio00} for $\alpha$-mixing, \cite{DouFerLan09} for $\eta$-dependence or \cite{DurTus12} for multiple mixing condtions, among others.  Here and throughout, weak convergence is understood in the sense of Definition~1.3.3 of \cite{VanWel96}.

In this note, we are interested in situations in which the sequence $(\vect U_i)_{i\in\Z}$ satisfies strong ($\alpha$-)mixing conditions. Let $(X_i)_{i\in\Z}$ be a sequence of Banach-space valued random variables. For $a\le b$, where $a,b\in\Z\cup\{-\infty,\infty\}$, let $\Fc_a^b$ denote the $\sigma$-field generated by $(X_i)_{a\le i \le b}$. The strong mixing coefficients of the sequence $(X_i)_{i \in \Z}$ are then defined by $\alpha_0 = 1/2$ and 
\[
\alpha_n = \sup_{p \in \Z} \sup_{A \in \Fc_{-\infty}^p,B\in \Fc_{p+n}^{+\infty}} | \Prob(A \cap B) - \Prob(A) \Prob(B) |, \qquad n \ge 1.
\]
The sequence $(X_i)_{i \in \Z}$ is said to be {\em strongly mixing} if $\alpha_n \to 0$ as $n \to \infty$. 

In the following, let $\alpha_n$ denote the mixing coefficients of the sequence $(\vect U_i)_{i\in\Z}$. It has been shown by \cite{Rio00} that $\Db_n \weak \Db$ in $(\ell^\infty([0,1]^d), \| \cdot \|_\infty)$ provided that $\alpha_n = O(n^{-a})$ for some $a>1$, thereby improving previous results by \cite{Yos75} and \cite{ShaYu96}.

In this note, we are interested in the slightly more general sequential empirical process 
\[
	\Bb_{n}(s,\vect u) = \frac{1}{\sqrt n}  \sum_{i=1}^{\ip{s n}}  \{ \ind(\vect U_i \le \vect u) - C(\vect u) \},
\]
where $(s,\vect u) = (s,u_1, \dots, u_d) \in [0,1]^{d+1}$ and $\ip{ns}$ denotes the integer part of $ns$. Note that $\Db_{n}(\vect u) =\Bb_{n}(1,\vect u)$. Investigating the process $\Bb_n$ is interesting for several reasons in mathematical statistics. For instance, the process can be used to derive nonparametric tests for change point detection in a $d$-dimensional time series, see, e.g., \cite{Ino01}. As a second application, suppose one is interested in constructing confidence bands for some real-valued estimator that can be written as a functional of the empirical cumulative distribution function, as for instance its integral over $\vect u\in[0,1]^d$. Then, following the self-normalizing approach developed in \cite{Sha10}, a weak convergence result for $\Bb_n$ can used to obtain confidence bands for this estimator that do not require a tuning-parameter-dependent estimation of the asymptotic covariance.

Regarding weak convergence results for $\Bb_n$, it is again known for various types of weak dependence conditions that 
\begin{align} \label{eq:weak}
	\Bb_n \weak \Bb_C \quad \text{ in } (\ell^\infty([0,1]^{d+1}, \| \cdot \|_\infty),
\end{align}
where $\Bb_C$ denotes a tight, centered Gaussian process with covariance
\begin{align*}
	 \Cov\{  \Bb_C(s,\vect u) , \Bb_C(t,\vect v) \} 
	= (s \wedge t) \sum_{i\in\Z} \Cov\{ \ind(\vect U_0\le \vect u) , \ind(\vect U_i\le \vect v) \},
\end{align*}
see for instance \cite{DedMerRio13} for $\beta$-mixing or \cite{DehDurTus13} for multiple mixing properties, among others. To the best of our knowledge, the best rate available in the literature for strongly mixing sequences follows from  the strong approximation result established in \cite{Dho84}: if $\alpha_n=O(n^{-b})$ with $b>2+d$, then \eqref{eq:weak} holds.
It is the purpose of the present note to improve this rate to $\alpha_n=O(n^{-a})$ for any $a>1$, independently of the dimension $d$, which is the same rate as established in \cite{Rio00} for the non-sequential process $\Db_n$. The proof of this result is inspired by the proof of Theorem 2.12.1 in \cite{VanWel96} and is based
on an adapted version of Ottaviani's inequality, see Proposition A.1.1 in the last-named reference, to strongly mixing sequences. This inequality might be of independent interest in other applications where one wants to transfer a weak convergence result from the non-sequential to the sequential setting.

\section{Main result}
\def\theequation{2.\arabic{equation}}
\setcounter{equation}{0}

\begin{theorem} \label{theo:weak}
If $\alpha_n=O(n^{-(1+\eta)})$ for some $\eta>0$, then, as $n\to\infty$,
\begin{align*} 
	\Bb_n \weak \Bb_C \quad \text{ in } (\ell^\infty([0,1]^{d+1}, \| \cdot \|_\infty).
\end{align*}
\end{theorem}

For the proof of this Theorem, we need to establish weak convergence of the finite-dimensional distributions (fidis) and asymptotic tightness. Regarding weak convergence of the fidis, we can for instance apply Theorem 2.1 in \cite{Pel96}. The details are omitted for the sake of brevity.

Let us consider the tightness part. For some function $f\in\ell^\infty([0,1]^p)$, $p\ge 1$, and $\delta>0$ let
\begin{align*}
w_\delta(f) = \sup_{\| \vect x - \vect y \|\le \delta } |f(\vect x) - f (\vect y) |
\end{align*}
denote the modulus of continuity of $f$.  By the results in \cite{VanWel96}, Section 1.5, the following Lemma completes the proof of Theorem~\ref{theo:weak}.
\begin{lemma}[Asymptotic tightness of $\Bb_n$]
Let $\alpha_n=O(n^{-(1+\eta)})$ for some $\eta\in(0,1)$. Then
\[
	\lim_{\delta \searrow 0} \limsup_{n\to\infty} \Prob(w_\delta(\Bb_{n}) > \eps) = 0.
\]
\end{lemma}

\begin{proof}
First, note that, by the results in Section 7 in \cite{Rio00} and Theorem 1.5.7 and its addendum in \cite{VanWel96}, we have
\begin{align} \label{eq:tid}
	\lim_{\delta \searrow 0} \limsup_{n\to\infty} \Prob(w_\delta(\Db_{n}) > \eps) =0.
\end{align}
By the triangle inequality
\begin{multline}\label{eq:decom1}
w_\delta(\Bb_{n}) \le \sup_{|s_1 - s_2| \le \delta} \sup_{\vect u\in[0,1]^d} | \Bb_{n}(s_1,\vect u) - \Bb_{n}(s_2,\vect u) | \\
	+ \sup_{0 \le s\le 1} \sup_{\| \vect u - \vect v \| \le \delta} | \Bb_{n}(s,\vect u) - \Bb_{n}(s,\vect v) |.
\end{multline}
The second summand is equal to
\begin{multline*}
	\max_{k=1}^n \sup_{\| \vect u - \vect v \| \le \delta}| \Bb_{n}(k/n,\vect u) - \Bb_{n}(k/n,\vect v) | \\
		= \max_{k=1}^n \sup_{\| \vect u - \vect v \| \le \delta}\left| \sqrt{\frac{k}{n}} \{ \Db_{k}(\vect u) - \Db_{k}(\vect v)\}\right|  = \max_{k=1}^n \sqrt{\frac{k}{n}}\, w_\delta(\Db_{k}).
\end{multline*}
Define $G_i(\vect u,\vect v)=\ind(\vect U_i\le \vect u) - C(\vect u) - \ind(\vect U_i \le \vect v) + C(\vect v)$. Set $\kappa=\eta/8$ and  $\ell_n=\ip{n^{1/2-\kappa}}$. For instance by observing that both $\Db_n :\Omega \to D([0,1]^d)$ and $w_\delta: D([0,1]^d) \to \R$ are ball-measurable, we can apply Ottaviani's inequality under strong mixing, see Lemma \ref{lem:ott}, with  $T = \{ t=(\vect u,\vect v) \in [0,1]^{2d} : \| \vect u - \vect v \| \le \delta\}$ and $Y_i(t) = n^{-1/2} G_i(t)$, where $t=(\vect u,\vect v)$. Let $\eps>0$. Then we obtain
\begin{align} \label{eq:decom2}
\Prob( \max_{k=1}^n \sqrt{k/n}\, w_\delta(\Db_{k}) > 3 \eps) 
	\le \frac{A_{n1} + A_{n2} + \ip{n/\ell_n} \alpha_{\ell_n}}{1- \max_{k=1}^n \Prob(\sqrt{k/n} \, w_\delta(\Db_{k}) > \eps) },
\end{align}
where $A_{n1}=\Prob(w_\delta(\Db_{n})> \eps)$ and 
\begin{align*} 
	A_{n2} &= \Prob\bigg(\max_{ \genfrac{}{}{0pt}{}{ j < k\in\{1, \dots, n\} }{ k-j \le 2\ell_n } } \sup_{\|\vect u - \vect v\| \le \delta} \frac{1}{\sqrt n} \bigg| \sum_{i=j+1}^{k} G_i(\vect u, \vect v)\bigg|> \eps\bigg). 
\end{align*}
For sufficiently large $n$, we have $\ell_n \ge \tfrac{1}{2} n^{1/2 - \kappa} $, whence  
\begin{align*}
	\ip{n/\ell_n} \alpha_{\ell_n} \le  \tfrac{1}{2} n^{1-(1/2-\kappa) - (1/2-\kappa)(1+\eta)} = \tfrac{1}{2}  n^{2\kappa - \eta/2 + \kappa \eta}= \tfrac{1}{2} n^{\eta/4(\eta/2-1)} = o(1)
\end{align*}
as $n\to\infty$.
Next, $A_{n1}$ converges to $0$ as $n\to \infty$ followed by $\delta\searrow 0$ by \eqref{eq:tid}. Moreover,
\[
	\max_{ \genfrac{}{}{0pt}{}{ j < k\in\{1, \dots, n\} }{ k-j \le 2\ell_n }} \sup_{\|\vect u - \vect v\| \le \delta} \frac{1}{\sqrt n} \bigg| \sum_{i=j+1}^{k} G_i(\vect u, \vect v)\bigg| \le 8 \ell_n/\sqrt{n} \le 8 n^{-\kappa} =o(1),
\]
as $n\to\infty$, whence $A_{n2} = o(1)$.

To complete the treatment of the second summand in \eqref{eq:decom1}, it remains to be shown that the denominator in \eqref{eq:decom2} is bounded away from zero for sufficiently large $n$ and small $\delta$. By \eqref{eq:tid}, there exists $\delta_0 > 0$ such that $\limsup_{n\to\infty} \Prob( w_\delta(\Db_n) > \eps) < 1/2$ for all $\delta \le \delta_0$. Then, there exists $n_0=n_0(\delta_0)$ such that $\Prob( w_{\delta_0}(\Db_n) > \eps)< 1/2$ for all $n \ge n_0$.
Therefore, for all $\delta<\delta_0$,
\[
	 \max_{k=n_0}^n \Prob(\sqrt{k/n}\, w_{\delta}(\Db_{k})>\eps) \le \max_{k=n_0}^n \Prob( w_\delta(\Db_{k})>\eps) \le \max_{k=n_0}^n \Prob( w_{\delta_0}(\Db_{k})>\eps)    < 1/2.
\]
On the other hand, for $k<n_0$ and arbitrary $\delta>0$, we have 
\[
	w_\delta(\Db_k) \le 2 \sup_{\vect u \in [0,1]^d} | \Db_k(\vect u)| \le 4 \sqrt k \le 4 \sqrt n_0,
\]
which implies that  $\max_{k=1}^{n_0-1} \Prob(\sqrt{k/n}\, w_{\delta}(\Db_{k})>\eps)=0$ for sufficiently large $n$ and all $\delta>0$. Therefore, the denominator in \eqref{eq:decom2} is bounded from below by $1/2$.

Finally, consider the first suprema on the right of \eqref{eq:decom1}. It suffices to show that, for every $\eps>0$,
\[
	\Prob\left( \max_{\genfrac{}{}{0pt}{}{ 0 \le j\delta \le 1 }{ j\in  \N }  } \sup_{j\delta \le s \le (j+1)\delta} \sup_{\vect u \in [0,1]^d} | \Bb_{n} (s,\vect u) - \Bb_{n}(j\delta, \vect u) | > 3\eps \right)
\]
converges to $0$ as $n\to \infty$ followed by $\delta\searrow 0$. By stationarity of the increments of $\Bb_{n}$ in $s$, the at most $\lceil 1/\delta \rceil$ terms in the maximum are identically distributed. Therefore, the probability can be bounded by
\begin{align} \label{eq:decom3}
	&\ \lceil 1/\delta \rceil \, \Prob\left( \sup_{0 \le s \le \delta } \sup_{\vect u \in [0,1]^d } | \Bb_{n}(s, \vect u) | > 3\eps \right) \nonumber \\
	= & \ \lceil 1/\delta \rceil \, \Prob\left( \max_{k =1}^{ \ip{n \delta} } \sup_{\vect u \in [0,1]^d } | \sqrt{k/n}\, \Db_{k}(\vect u) | > 3\eps \right) \nonumber \\
	\le & \  \frac{\lceil 1/\delta \rceil  (B_{n1}+B_{n2} + \ip{ \ip{n\delta} /\ell_n} \alpha_{\ell_n} ) }{1-  \max_{k =1}^{ \ip{n \delta} } \Prob\left( \sqrt{k/n}\, \sup_{\vect u \in [0,1]^d} | \Db_{k}(\vect u) | > \eps \right) },
\end{align}
by the Ottaviani-type inequality in Lemma~\ref{lem:ott}, where
\begin{align*}
	B_{n1} &= \Prob\left( \sqrt{ \ip{n\delta}/n} \, \sup_{\vect u \in [0,1]^d} |\Db_{ \ip{n\delta}} (\vect u) | > \eps \right) \\
	B_{n2} &= \Prob\left( \max_{ \genfrac{}{}{0pt}{}{j < k \in \{1, \dots,  \ip{n \delta} \} }{ k-j \le 2\ell_n} } \sup_{\vect u \in [0,1]^d} \left|n^{-1/2} \sum_{i=j+1}^k  \{ \ind(\vect U_i \le \vect u) - C(\vect u) \} \right| > \eps \right).
\end{align*}
Here, we were allowed to apply Lemma \ref{lem:ott} by a similar argument as before.
It remains to be shown that $B_{n1}$, $B_{n2}$ and $\ip{ \ip{n\delta} /\ell_n} \alpha_{\ell_n}$, multiplied with $1/\delta$, converges to zero as $n\to \infty$ followed by $\delta \searrow0$ and that the denominator in \eqref{eq:decom3} is bounded away from zero. First of all, for $\delta \le 1$, we have
\[
	\lceil 1/\delta \rceil \ip{ \ip{n\delta} /\ell_n} \alpha_{\ell_n} \le \left(\frac{1}{\delta} + 1\right) \frac{n \delta \alpha_{\ell_n}}{\ell_n} \le 2\frac{ n\alpha_{\ell_n}}{\ell_n} = o(1)
\]
as $n\to\infty$, by the same arguments as above. 
Second, by the Portmanteau-Theorem, 
\begin{align*}
 	\limsup_{n\to\infty} B_{n1} 
	&\le \limsup_{n\to\infty} \Prob( \sup_{\vect u \in [0,1]^d} |\Db_n(\vect u)| > \eps\delta^{-1/2} ) \\
	&\le \limsup_{n\to\infty} \Prob( \sup_{\vect u \in [0,1]^d } |\Db_n(\vect u)| \ge \eps\delta^{-1/2} ) \\
	&\le \Prob( \sup_{\vect u \in [0,1]^d } |\Db_C(\vect u)| \ge \eps\delta^{-1/2} ) 
	= 	\Prob( \sup_{\vect u \in [0,1]^d} | \Db_C(\vect u) | > \eps \delta^{-1/2} )
\end{align*}
since $\sup_{\vect u \in [0,1]^d} | \Db_C(\vect u) | $ is a continuous random variable.
Since additionally $\sup_{\vect u \in [0,1]^d} | \Db_{C}(\vect u) |$ possesses moments of any order (cf.~Proposition A.2.3 in \cite{VanWel96}), the latter probability converges to zero faster than any power of $\delta$, as $\delta \searrow 0$. 

Third, regarding $B_{n2}$, we have
\begin{align*}
	\max_{ \genfrac{}{}{0pt}{}{j < k \in \{1, \dots,  \ip{n \delta} \} }{ k-j \le 2\ell_n}  } \sup_{\vect u \in [0,1]^d} \left|n^{-1/2} \sum_{i=j+1}^k  \{ \ind(\vect U_i \le \vect u) - C(\vect u) \} \right|   \le  4 \ell_n n^{-1/2}= o(1),
\end{align*}
as $n\to \infty$, whence $B_{n2}$ converges to zero as $n\to \infty$ followed by $\delta \searrow0$ as asserted.

Finally, let us consider the denominator in \eqref{eq:decom3}. By a similar argument as before, we have from the Portmanteau Theorem that
\[
	\limsup_{n\to\infty} \Prob( \sup_{\vect u \in [0,1]^d} |\Db_n(\vect u)| > \eps ) \le  \Prob( \sup_{\vect u \in [0,1]^d} |\Db_C(\vect u)| > \eps )
\]
by continuity of $\sup_{\vect u \in [0,1]^d} |\Db_C(\vect u)|$. Also, since $\eps>0$, we obtain  that $p=\Prob( \sup_{\vect u \in [0,1]^d} |\Db_C(\vect u)| > \eps )<1$, whence we can choose $\zeta>0$ such that
$0<\zeta< 1-p$. It follows that there exists $n_0\in \N$ such that
\[
\sup_{k \ge n_0} \Prob( \sup_{\vect u \in [0,1]^d } |\Db_k(\vect u)| > \eps ) \le p + \zeta.
\]
Hence,
\[
	 \max_{k =n_0}^{ \ip{n \delta} } \Prob( \sqrt{k/n} \sup_{\vect u \in [0,1]^d} |\Db_k(\vect u)| > \eps )
\leq \max_{k =n_0}^{ \ip{n \delta} } \Prob( \sup_{\vect u \in [0,1]^d} |\Db_k(\vect u)| > \eps ) \le p+\zeta.
\]
On the other hand, for $k < n_0$, we have  $\sqrt{k/n} \sup_{\vect u \in [0,1]^d} |\Db_k(\vect u)| \le \sqrt{k/n} \times 2\sqrt k \leq 2 n_0 / \sqrt{n}$, which implies that, for n large enough,
\[
 \max_{k =1}^{ n_0-1 } \Prob( \sqrt{k/n} \sup_{\vect u \in [0,1]^d} |\Db_k(\vect u)| > \eps ) = 0.
\]
Hence, the denominator in  \eqref{eq:decom3} is bounded from below by $1- p - \zeta>0$ for n large enough. This completes the proof.
\end{proof}

\section{An auxiliary Lemma} \label{sec:aux}
\def\theequation{3.\arabic{equation}}
\setcounter{equation}{0}


Let $(X_i)_{i\in\Z}$ be a sequence of random elements in some Banach space~$\Eb$. Let $T$ be some arbitrary index set and, for $i\in\Z$, let $G_i\in\ell^\infty(\Eb \times T)$. For $t\in T$, set $Y_i(t) = G_i(X_i,t)$ and $S_n(t)=\sum_{i=1}^n Y_i(t)$, for $n\ge 1$ and $S_0 \equiv 0$.
Finally, for $f\in\ell^\infty(T)$, let $\| f\| = \sup_{t \in T} | f(t) |$.
\begin{lemma} [An Ottaviani-type inequality under strong mixing] \label{lem:ott}
Suppose that $\| S_m - S_n \|$ is measurable for each $0 \le n < m$. 
Then, for each $\eps>0$ and $1 \le \ell <n$,
\begin{multline*}
	\Prob\left( \max_{k=1}^n \| S_k \| > 3\eps \right)  \times \left\{ 1- \max_{k=1}^n \Prob\left( \| S_n - S_k\| > \eps\right) \right\} \\
		\le \Prob( \| S_n \| > \eps) + \Prob\left(\max_{ \genfrac{}{}{0pt}{}{j < k \in \{1, \dots, n \} }{ k-j \le 2\ell}  }  \| S_{k} - S_j \| > \eps\right) + \ip{n/\ell} \times \alpha_\ell,
\end{multline*}
where $\alpha_n$ denotes the sequence of mixing coefficients of the sequence $(X_i)_{i\in\Z}$.
\end{lemma}


\begin{proof}
For $k=1, \dots, n$, define the event $B_k$ by
\[
	B_k = \left\{ \| S_k \| > 3 \eps, \| S_1 \| \le 3\eps, \dots, \|S_{k-1} \| \le 3\eps \right\}.
\]
Note that these events are pairwise disjoint and that their union is given by $\{\max_{k=1}^n \| S_k \| > 3\eps\}$. Furthermore, for $m=1, \dots, \ip{n/\ell}-1$, let 
\[
	C_m= \bigcup_{k=(m-1)\ell + 1}^{m\ell} B_k, \quad \text{ and }\ C_{\ip{n/\ell}} = \bigcup_{k=(\ip{n/\ell}-1)\ell+1}^n B_k,
\]
which are also pairwise disjoint and have the same union as the $B_k$'s.

Now, let us first consider a fixed $m \le \ip{n/\ell}-1$. Then
\begin{align} \label{eq:dd}
	&\ \Prob\left( C_m \right)  \times  \min_{k=1}^n \Prob\left( \| S_n - S_k\| \le \eps\right) \nonumber \\
		\le &\ \Prob\left(C_m\right) \Prob\left( \| S_n - S_{(m+1)\ell}\| \le \eps\right) \nonumber \\
		\le &\ \Prob\left( C_m, \| S_n - S_{(m+1)\ell}\| \le \eps\right) +  \alpha_\ell \nonumber \\
		\le &\ \Prob\left(C_m, \max_{k=(m-1)\ell+1}^{m\ell} \| S_{k} \| > 3 \eps, \| S_n - S_{(m+1)\ell}\| \le \eps\right) +  \alpha_\ell. 
\end{align}
Since $\|S_k \| \le \| S_k - S_{(m+1)\ell}\|  + \|S_{(m+1)\ell} - S_n \|+ \|S_n\|$ for any $k=1, \dots, n$, 
we have
\begin{align*}
	\| S_n \| & \ge \max_{k=(m-1)\ell+1}^{m\ell} \{ \| S_k \| - \|  S_{(m+1)\ell} - S_k \| - \|S_{(m+1)\ell} - S_n \| \} \\
	& \ge\bigl\{ \max_{k=(m-1)\ell+1}^{m\ell}  \| S_k \| \bigr\} - \bigl\{ \max_{k=(m-1)\ell+1}^{m\ell} \|  S_{(m+1)\ell} - S_k \| \bigr\} - \|S_{(m+1)\ell} - S_n \|. 
\end{align*}
Therefore, we can estimate the right-hand side of \eqref{eq:dd} by
\begin{align*}
		\le &\ \Prob \left(C_m, \|S_n\| > 2 \eps -  \max_{k=(m-1)\ell+1}^{m\ell} \| S_{(m+1)\ell} - S_{k} \| \right) +  \alpha_\ell \\
		\le &\ \Prob\left(C_m, \| S_n \| > \eps \right) + \Prob\left(C_m, \max_{k=(m-1)\ell+1}^{m\ell} \| S_{(m+1)\ell} - S_{k} \| > \eps\right) +  \alpha_\ell \\
		\le &\ \Prob\left(C_m, \| S_n \| > \eps \right) + \Prob\left(C_m, \max_{ \genfrac{}{}{0pt}{}{j < k \in \{1, \dots, n \} }{ k-j \le 2\ell}  }  \| S_{k} - S_j \| > \eps\right) +  \alpha_\ell.
\end{align*}

Next, let us consider the case $m = \ip{n/\ell}$. Then
\begin{align*}
	&\ \Prob\left( C_{\ip{n/\ell}} \right)  \times  \min_{k=1}^n \Prob\left( \| S_n - S_k\| \le \eps\right)  \\
		\le &\ \Prob\left(C_{\ip{n/\ell}}, \max_{k=(\ip{n/\ell}-1)\ell+1}^n \| S_{k} \| > 3 \eps \right) \\
		\le &\ \Prob \left(C_{\ip{n/\ell}}, \|S_n\| > 3 \eps - \max_{k=(\ip{n/\ell}-1)\ell+1}^n \| S_{n} - S_k\| \right)  \\
		\le &\ \Prob\left(C_{\ip{n/\ell}}, \| S_n \| > 2\eps \right) + \Prob\left(C_{\ip{n/\ell}},  \max_{k=(\ip{n/\ell}-1)\ell+1}^n \| S_{n} - S_k\| > \eps\right)  \\
		\le &\ \Prob\left(C_{\ip{n/\ell}}, \| S_n \| > \eps \right) + \Prob\left(C_{\ip{n/\ell}},  \max_{ \genfrac{}{}{0pt}{}{j < k \in \{1, \dots, n \} }{ k-j \le 2\ell}  }  \| S_{k} - S_j \| > \eps\right) .
\end{align*}

Joining both cases, we have, for any $m=1, \dots, \ip{n/\ell}$,
\begin{multline*}
	 \Prob\left( C_m \right)  \times  \min_{k=1}^n \Prob\left( \| S_n - S_k\| \le \eps\right)   \\
	 \le \Prob\left(C_m, \| S_n \| > \eps \right) +  \Prob\left(C_m, \max_{ \genfrac{}{}{0pt}{}{j < k \in \{1, \dots, n \} }{ k-j \le 2\ell}  }  \| S_{k} - S_j \| > \eps\right) +  \alpha_\ell.
\end{multline*}
Summation over $m$ finally yields the assertion.
\end{proof}

\textbf{Acknowledgements.} 
The author is thankful to \textit{Ivan Kojadinovic} for thorough proofreading and numerous suggestions concerning this manuscript. 


This work has been supported in parts by the Collaborative Research Center ``Statistical modeling of nonlinear dynamic processes'' (SFB 823) of the German Research Foundation (DFG) and by the IAP research network Grant P7/06 of the Belgian government (Belgian Science Policy), which is gratefully acknowledged.

\bibliographystyle{chicago}
\bibliography{biblio}
\end{document}